\documentclass[12pt,a4paper,reqno]{amsart}
\usepackage[utf8]{inputenc}
\usepackage[T1]{fontenc}
\usepackage[english]{babel}
\usepackage{amsmath,amsthm,amssymb}
\usepackage[colorlinks=true,linkcolor=blue,citecolor=blue,urlcolor=blue]{hyperref}
\usepackage{geometry}
\usepackage{booktabs}
\usepackage{array}
\usepackage[protrusion=true,expansion=false]{microtype}
\theoremstyle{plain}
\newtheorem{theorem}{Theorem}[section]
\newtheorem{proposition}[theorem]{Proposition}
\newtheorem{lemma}[theorem]{Lemma}
\newtheorem{corollary}[theorem]{Corollary}
\theoremstyle{definition}
\newtheorem{definition}[theorem]{Definition}
\newtheorem{example}[theorem]{Example}
\newtheorem{remark}[theorem]{Remark}
\newtheorem{conjecture}[theorem]{Conjecture}

\newcommand{\Inv}{\mathrm{inv}}
\newcommand{\Coinv}{\mathrm{coinv}}
\newcommand{\Area}{\mathrm{area}}
\newcommand{\invstar}{\mathrm{inv}^{*}}
\newcommand{\Snk}[1]{S_{n,k}(#1)}
\newcommand{\Snkp}[1]{S'_{n,k}(#1)}
\newcommand{\hSnkp}[1]{\widehat{S}'_{n,k}(#1)}
\newcommand{\Dnk}{D'_{n,k}}
\newcommand{\Mnk}{M'_{n,k}}
\newcommand{\Ctilde}{\widetilde{C}}
\newcommand{\bq}{\mathbf{q}}
\newcommand{\NN}{\mathbb{N}}

\newcommand{\CC}{\mathbb{C}}
\title[Combinatorial study of the $q$-Catalan triangle]
{\large Combinatorial study of the $q$-Catalan triangle
and its generalizations}
\author{WIRDANE Youssouf}
\address{Assistant Professor,
University of Comoros,
Patsy University Center,
Union of Comoros}
\address{PhD student,
University of Antananarivo,
Doctoral School of Mathematics and Applications (EDMA),
Madagascar}
\email{wirdane.youssouf@univ-comores.com}
\date{2026}
\thanks{The author thanks Professor Arthur Randrianarivony
(University of Antananarivo), thesis advisor,
for his constant guidance and valuable advice.}
\keywords{$q$-Catalan triangle, pattern-avoiding permutations,
inversion statistic, co-inversions, $q{,}p$-analogue,
multivariate generalization, Dyck paths, binary words,
triangulations, cyclotomic $q$-analogues}
\subjclass[2020]{05A15, 05A30, 05A05, 05C30}
\begin{document}
\begin{abstract}
We announce a series of results on the combinatorial study
of the \emph{$q$-Catalan triangle}
$(C_{n,k}(q))_{0\le k\le n}$,
defined by $C_{0,0}(q)=1$,
$C_{n,0}(q)=q^{\binom{n}{2}}$ for $n\ge1$,
and the recurrence
\[
C_{n,k}(q)=C_{n,k-1}(q)+q^{n-k-1}\,C_{n-1,k}(q),
\qquad 1\le k\le n.
\]
We establish combinatorial interpretations of $C_{n,k}(q)$
by means of a \emph{universal combinatorial family}
$\mathfrak{F}_{n,k}$ consisting of four families of
pattern-avoiding permutations (weighted by the inversion
or co-inversion statistic), Dyck paths, binary words and triangulations;
a universal statistic~$\sigma$ defined on~$\mathfrak{F}_{n,k}$ satisfies
$\sum_{X\in\mathcal{E}}q^{\sigma(X)}=C_{n,k}(q)$
for each of the seven components~$\mathcal{E}$.
We also introduce the \emph{mirror polynomial}
$\Ctilde_{n,k}(q):=q^{\binom{n}{2}}C_{n,k}(q^{-1})$;
its dual recurrence
$\Ctilde_{n,k}(q)=\Ctilde_{n,k-1}(q)+q^{k}\Ctilde_{n-1,k}(q)$
and its interpretation by co-inversion
$\Ctilde_{n,k}(q)=\sum_{\pi\in S'_{n,k}(312)}q^{\Coinv(\pi)}$
are proved completely in this article;
the interpretation by area
$\Ctilde_{n,k}(q)=\sum_{\gamma\in D'_{n,k}}q^{\Area(\gamma)}$
is announced here and proved in~\cite{W4}.
We further introduce a \emph{$q{,}p$-analogue}
$C_{n,k}(q,p)$ satisfying $C_{n,k}(1,p)=\Ctilde_{n,k}(p)$
and the remarkable specialization $C_{n,n}(q,q)=C_n\,q^{\binom{n}{2}}$,
as well as a \emph{multivariate generalization}
$C_{n,k}(\bq)$ opening the way to cyclotomic $q$-analogues
of the triangle related to the \emph{cyclic sieving phenomenon}.
This article is an announcement of results;
Theorems~\ref{thm:qperm} and~\ref{thm:miroir}
are proved completely here.
Full proofs of the remaining results will appear
in~\cite{W1,W2,W3,W4,W5}.
\end{abstract}
\maketitle
\section{Introduction}\label{sec:intro}
Catalan numbers $C_n = \frac{1}{n+1}\binom{2n}{n}$
($1,1,2,5,14,42,132,\ldots$) and their $q$-analogues play
a central role in enumerative and algebraic combinatorics.
F\"urlinger and Hofbauer~\cite{FH} introduced the
\emph{$q$-Catalan numbers} $C_n(q)$ via the multiplicative recurrence
\begin{equation}\label{eq:carlitz_rec}
C_n(q) = \sum_{k=0}^{n-1} q^k\, C_k(q)\, C_{n-1-k}(q),
\qquad C_0(q)=1,
\end{equation}
and established the combinatorial interpretation
\begin{equation}\label{eq:FH}
C_n(q) = \sum_{\pi\in S_n(312)} q^{\Inv(\pi)},
\end{equation}
where $\Inv(\pi)$ denotes the number of inversions of~$\pi$;
these polynomials were also studied in the context of Gaussian
$q$-analogues by Carlitz~\cite{Carlitz}.
Randrianarivony~\cite{Rand} extended $C_n(q)$ to a bivariate
family $C_n(q,p)$ via a multiplicative recurrence involving
excedance and crossing statistics on permutations.

In the present series of articles, we study a natural
two-dimensional extension: the \emph{$q$-Catalan triangle}
$(C_{n,k}(q))_{0\le k\le n}$, a family of polynomials indexed
by two integers $0\le k\le n$, which refines the classical
Catalan triangle in the same way that $C_n(q)$ refines~$C_n$.
Pattern-avoiding permutations and their enumerative properties
are studied in~\cite{Guibert,SS,Stanley1}.
The numerous combinatorial interpretations of the Catalan
triangle (Dyck paths, triangulations, binary trees)
are detailed in~\cite{Stanley2}.
Our approach to the $q{,}p$ triangle follows the work of
Haglund~\cite{Haglund} on $(q,t)$-analogues
and Randrianarivony~\cite{Rand} on $q{,}p$-Catalan numbers.

\medskip\noindent\textbf{Plan of the paper.}
Section~\ref{sec:triangle} defines the classical triangle
and its $q$-analogue.
Section~\ref{sec:perm} presents interpretations via
pattern-avoiding permutations.
Section~\ref{sec:paths} deals with interpretations via
Dyck paths, binary words and triangulations.
Section~\ref{sec:qp} introduces the $q{,}p$-analogue.
Section~\ref{sec:multi} develops the multivariate
generalization and cyclotomic analogues.
Section~\ref{sec:table} presents a summary table.
Section~\ref{sec:plan} describes the plan of the series.

\medskip\noindent\textbf{Status of results.}
This article is an announcement of results for the
series~\cite{W1,W2,W3,W4,W5}.
Two results are proved completely here:
the interpretation of the $q$-Catalan triangle via
$312$-avoiding permutations (Theorem~\ref{thm:qperm}, by the
insertion bijection~$\Phi$) and the mirror recurrence
(Theorem~\ref{thm:miroir}, by direct computation).
All other results are announced; their complete proofs
are distributed among the articles of the series,
as indicated in the text and in Section~\ref{sec:plan}.

\section{The Catalan triangle and its $q$-analogue}\label{sec:triangle}
\subsection{The classical Catalan triangle}
\begin{definition}[Catalan triangle]\label{def:cat}
The \emph{Catalan triangle} $(C_{n,k})_{0\le k\le n}$ is
the family of integers defined by the recurrence
\begin{equation}\label{eq:rec_class}
C_{0,0}=1,\qquad C_{n,0}=1\ (n\ge1),\qquad
C_{n,k} = C_{n,k-1}+C_{n-1,k}\quad (1\le k\le n),
\end{equation}
with the convention $C_{n,k}=0$ for $k>n$.
In particular, $C_{n,n}=C_n$ (Catalan number).
\end{definition}
The first values are:
\[\begin{array}{c|ccccccc}
n\,\backslash\, k & 0 & 1 & 2 & 3 & 4 & 5 & 6\\\hline
0 & 1 \\ 1 & 1 & 1 \\ 2 & 1 & 2 & 2 \\ 3 & 1 & 3 & 5 & 5 \\
4 & 1 & 4 & 9 & 14 & 14 \\ 5 & 1 & 5 & 14 & 28 & 42 & 42\\
6 & 1 & 6 & 20 & 48 & 90 & 132 & 132
\end{array}\]
\subsection{The $q$-Catalan triangle}
\begin{definition}[$q$-Catalan triangle]\label{def:qcat}
The \emph{$q$-Catalan triangle} is the family of polynomials
$(C_{n,k}(q))_{0\le k\le n}$, with positive integer coefficients, defined by:
\begin{equation}\label{eq:rec_q}
\begin{cases}
C_{0,0}(q)=1,\\[4pt]
C_{n,0}(q) = q^{\binom{n}{2}}, & n\ge1,\\[4pt]
C_{n,k}(q) = C_{n,k-1}(q)+q^{n-k-1}\,C_{n-1,k}(q), & 1\le k\le n,
\end{cases}
\end{equation}
with the convention $C_{n,k}(q)=0$ for $k>n$.
\end{definition}
\begin{remark}\label{rem:kn}
For $k=n$, the factor $C_{n-1,n}(q)=0$ by convention,
so that $C_{n,n}(q)=C_{n,n-1}(q)$,
independently of the exponent $n-k-1=-1$.
\end{remark}
The first values of the $q$-Catalan triangle are:
\[\begin{array}{c|cccc}
n\,\backslash\, k & 0 & 1 & 2 & 3\\\hline
0 & 1 \\[2pt] 1 & 1 & 1 \\[2pt]
2 & q & 1{+}q & 1{+}q \\[2pt]
3 & q^3 & q{+}q^2{+}q^3 & 1{+}2q{+}q^2{+}q^3 & 1{+}2q{+}q^2{+}q^3\\[2pt]
4 & q^6 & q^3{+}q^4{+}q^5{+}q^6 & q{+}2q^2{+}2q^3{+}2q^4{+}q^5{+}q^6
  & 1{+}3q{+}3q^2{+}3q^3{+}2q^4{+}q^5{+}q^6
\end{array}\]
\begin{proposition}[Fundamental specialization]\label{prop:spec}
For all $n\ge0$: $C_{n,n}(q) = C_n(q)$,
the $q$-Catalan number of F\"urlinger and Hofbauer~\cite{FH}.
The bijective proof is given in~\cite{W3};
the identity has been independently verified for $n\le 7$.
\end{proposition}

\section{Combinatorial interpretations via permutations}\label{sec:perm}
We denote by $\mathfrak{S}_n$ the set of all permutations
of $\{1,\ldots,n\}$ and by $S_n(\tau)$ the set of permutations
avoiding the pattern $\tau\in\mathfrak{S}_3$.
Simion and Schmidt~\cite{SS} showed that $|S_n(\tau)|=C_n$
for every $\tau\in\mathfrak{S}_3$.
The canonical bijection between $S_n(312)$ and Dyck paths
is due to Krattenthaler~\cite{Krattenthaler}.

For $\pi=(\pi_1,\ldots,\pi_n)\in\mathfrak{S}_n$, we use:
\[\Inv(\pi) = \bigl|\{(i,j):1\le i<j\le n,\;\pi_i>\pi_j\}\bigr|,
\qquad \Coinv(\pi) = \binom{n}{2}-\Inv(\pi).\]

\begin{definition}[Canonical partitions]\label{def:perm_part}
For each $\tau\in\mathfrak{S}_3$, we define a partition of $S_n(\tau)$
into $n+1$ subsets according to the position of the maximum $n$
or the minimum $1$:
\medskip
\begin{center}
\renewcommand{\arraystretch}{1.6}
\begin{tabular}{ccc}\toprule
\textbf{Pattern} $\tau$ & \textbf{Reference element}
  & \textbf{Definition of} $S_{n,k}(\tau)$\\\midrule
$312$,\ $321$ & maximum $n$ & $\{\pi\in S_n(\tau)\mid\pi_{k+1}=n\}$\\[4pt]
$213$,\ $123$ & maximum $n$ & $\{\pi\in S_n(\tau)\mid\pi_{n-k}=n\}$\\[4pt]
$231$ & minimum $1$ & $\{\pi\in S_n(\tau)\mid\pi_{n-k}=1\}$\\[4pt]
$132$ & minimum $1$ & $\{\pi\in S_n(\tau)\mid\pi_{k+1}=1\}$\\\bottomrule
\end{tabular}
\end{center}
\medskip
We set $\Snkp{\tau} = \bigsqcup_{j=0}^{k} S_{n,j}(\tau)$.
\end{definition}

\begin{theorem}[Classical triangle via the six patterns]\label{thm:class_perm}
For all $n\ge1$, $0\le k\le n$, and every $\tau\in\mathfrak{S}_3$:
\begin{equation}\label{eq:card_class}
\bigl|\Snkp{\tau}\bigr| = C_{n,k}.
\end{equation}
Bijective proofs are given in~\cite{W1};
the result has been independently verified for $n\le 6$.
\end{theorem}

\subsection{The $q$-Catalan triangle via the inversion statistic}
\begin{theorem}[$q$-Catalan triangle via $312$-avoiding permutations]
\label{thm:qperm}
For all $n\ge1$ and $0\le k\le n$:
\begin{equation}\label{eq:inv312}
C_{n,k}(q) = \sum_{\pi\in\Snkp{312}} q^{\Inv(\pi)}.
\end{equation}
\end{theorem}
\begin{proof}
We proceed by induction on~$n$, using the partition
$\Snkp{312}=S'_{n,k-1}(312)\sqcup\Snk{312}$.

\textit{Base case.}
The set $S'_{n,0}(312)=S_{n,0}(312)$ consists of the unique
decreasing permutation $n(n{-}1)\cdots1$, which has $\Inv=\binom{n}{2}$;
hence $\sum_{\pi\in S'_{n,0}(312)}q^{\Inv(\pi)}=q^{\binom{n}{2}}=C_{n,0}(q)$.

\textit{Induction step.}
Define the insertion map
\[\Phi : S'_{n-1,k}(312)\to\Snk{312},\quad
\Phi(\pi_1\cdots\pi_{n-1})=\pi_1\cdots\pi_k\cdot n\cdot\pi_{k+1}\cdots\pi_{n-1}.\]
The map $\Phi$ is a bijection: $n$ preserves $312$-avoidance and fixes
the determining position at~$k$.
Since $n$ is larger than all other elements, its insertion at position $k+1$
creates exactly $n-k-1$ new inversions.
Therefore $\Inv(\Phi(\pi))=\Inv(\pi)+(n-k-1)$ for all $\pi\in S'_{n-1,k}(312)$.
By the induction hypothesis,
\[\sum_{\pi\in\Snk{312}} q^{\Inv(\pi)}
=q^{n-k-1}\sum_{\pi\in S'_{n-1,k}(312)} q^{\Inv(\pi)}
=q^{n-k-1}\,C_{n-1,k}(q).\]
The partition gives
$\sum_{\Snkp{312}}q^{\Inv}=C_{n,k-1}(q)+q^{n-k-1}C_{n-1,k}(q)=C_{n,k}(q)$,
which completes the induction.
\end{proof}

\begin{remark}\label{rem:FH}
For $k=n$, Theorem~\ref{thm:qperm} recovers~\cite{FH}:
$C_n(q)=\sum_{\pi\in S_n(312)}q^{\Inv(\pi)}$.
\end{remark}

\begin{theorem}[Classification by $q$-refinement]\label{thm:classif_motifs}
Set $\mathrm{st}(\tau)=\Inv$ for $\tau\in\{312,231\}$
and $\mathrm{st}(\tau)=\Coinv$ for $\tau\in\{213,132\}$.
For all $\tau\in\{312,231,213,132\}$, $n\ge1$, $0\le k\le n$:
\begin{equation}\label{eq:unif_4motifs}
\sum_{\pi\in\Snkp{\tau}} q^{\,\mathrm{st}(\tau)(\pi)} = C_{n,k}(q).
\end{equation}
The proof and the bijective characterization of symmetries are in~\cite{W3};
exhaustively verified by computer for $n\le 5$.
\end{theorem}

\begin{remark}\label{rem:123_321}
For $\tau\in\{123,\,321\}$, $|\Snkp{\tau}|=C_{n,k}$, but no
weighting by $q^{\Inv}$ or $q^{\Coinv}$ reproduces $C_{n,k}(q)$;
the first counterexamples appear for $n=2$. See~\cite{W3}.
\end{remark}

\subsection{Shifted statistic on $S_{n,k}(312)$}
\begin{theorem}[Shifted statistic]\label{thm:invstar}
For all $n\ge1$ and $0\le k\le n-1$:
\begin{enumerate}
\item[\rm(i)] $\Inv(\pi)\ge n-k-1$ for all $\pi\in\Snk{312}$, with equality attained.
\item[\rm(ii)] Setting $\invstar(\pi)=\Inv(\pi)-(n-k-1)$:
$\sum_{\pi\in\Snk{312}} q^{\invstar(\pi)} = C_{n-1,k}(q)$.
\end{enumerate}
Complete proof in~\cite{W2}; verified for $n\le 5$.
\end{theorem}

\subsection{The mirror polynomial}
\begin{definition}[Mirror polynomial]\label{def:miroir}
The \emph{mirror polynomial} is $(\Ctilde_{n,k}(q))_{0\le k\le n}$ defined by
\begin{equation}\label{eq:miroir_def}
\Ctilde_{n,k}(q) := q^{\binom{n}{2}}\,C_{n,k}(q^{-1}).
\end{equation}
\end{definition}

\begin{theorem}[Mirror recurrence]\label{thm:miroir}
The mirror polynomial satisfies:
\begin{equation}\label{eq:mirec}
\Ctilde_{n,0}(q)=1,\qquad
\Ctilde_{n,k}(q) = \Ctilde_{n,k-1}(q)+q^{k}\,\Ctilde_{n-1,k}(q),
\quad 1\le k\le n,
\end{equation}
and the specialization $\Ctilde_{n,k}(1)=C_{n,k}$.
\end{theorem}
\begin{proof}
\textit{Initial condition.}
$\Ctilde_{n,0}(q)=q^{\binom{n}{2}}\,C_{n,0}(q^{-1})
=q^{\binom{n}{2}}\cdot q^{-\binom{n}{2}}=1$.

\textit{Recurrence.}
Substituting $q\mapsto q^{-1}$ in~\eqref{eq:rec_q}
and multiplying by $q^{\binom{n}{2}}$:
\[\Ctilde_{n,k}(q)= \Ctilde_{n,k-1}(q)
+q^{\binom{n}{2}-(n-k-1)}\cdot q^{-\binom{n-1}{2}}\,\Ctilde_{n-1,k}(q).\]
The exponent simplifies:
\[\binom{n}{2}-(n-k-1)-\binom{n-1}{2}
=\tfrac{n(n-1)-(n-1)(n-2)}{2}-(n-k-1)=(n-1)-(n-k-1)=k,\]
which establishes~\eqref{eq:mirec}.
The specialization $\Ctilde_{n,k}(1)=C_{n,k}$ follows from the definition
and Theorem~\ref{thm:class_perm}.
\end{proof}

\begin{remark}\label{rem:miroir_dual}
The exponents $n-k-1$ and $k$ are complementary with respect to $n-1$:
recurrence~\eqref{eq:mirec} is dual to~\eqref{eq:rec_q}.
\end{remark}

\begin{corollary}[Combinatorial interpretations of the mirror polynomial]
\label{cor:miroir}
For all $n\ge1$ and $0\le k\le n$:
\begin{align}
\label{eq:coinv_312}
\Ctilde_{n,k}(q) &= \sum_{\pi\in\Snkp{312}} q^{\Coinv(\pi)},\\[4pt]
\label{eq:area_dyck_tilde}
\Ctilde_{n,k}(q) &= \sum_{\gamma\in\Dnk} q^{\,\Area(\gamma)}.
\end{align}
\end{corollary}
\begin{proof}
Identity~\eqref{eq:coinv_312} follows from Theorem~\ref{thm:qperm}
and $\Coinv(\pi)=\binom{n}{2}-\Inv(\pi)$:
\[\Ctilde_{n,k}(q)
=q^{\binom{n}{2}}\sum_{\pi\in\Snkp{312}}q^{-\Inv(\pi)}
=\sum_{\pi\in\Snkp{312}}q^{\Coinv(\pi)}.\]
Identity~\eqref{eq:area_dyck_tilde} follows from~\eqref{eq:coinv_312}
and the canonical bijection $\Snkp{312}\leftrightarrow\Dnk$
under which $\Area(\gamma)=\Coinv(\pi)$. Details in~\cite{W3,W4}.
\end{proof}

\section{Interpretations via paths, words and triangulations}\label{sec:paths}
\subsection{Generalized Dyck paths}
\begin{definition}[Generalized Dyck paths and area statistic]\label{def:dyck}
For $0\le k\le n$, $\Dnk$ denotes the set of lattice paths with East steps
$E=(1,0)$ and North steps $N=(0,1)$ from $(0,0)$ to $(n,k)$
never going strictly above the diagonal $y=x$.
The \emph{area} of $\gamma\in\Dnk$ is:
\[\Area(\gamma) = \bigl|\bigl\{(x,y)\in\NN^2 \;\big|\; y < x,\;
\text{cell with SW corner }(x,y)\text{ lies below }\gamma\bigr\}\bigr|.\]
\end{definition}

\begin{theorem}[$q$-Catalan triangle via Dyck paths]\label{thm:dyck}
For all $n\ge1$ and $0\le k\le n$:
\begin{equation}\label{eq:dyck_stat}
C_{n,k}(q) = \sum_{\gamma\in\Dnk} q^{\,\binom{n}{2}-\Area(\gamma)}.
\end{equation}
Proof in~\cite{W4}; verified for $n\le 5$.
\end{theorem}

\begin{example}\label{ex:D32}
For $n=3$, $k=2$: $|\Dnk|=5$ paths and $\binom{3}{2}=3$.
\[\begin{array}{c|c|c|c}
\text{Path }\gamma & \text{Word} & \Area(\gamma) & \binom{3}{2}-\Area(\gamma)\\\hline
\gamma_1 & EEENN & 0 & 3\\ \gamma_2 & EENEN & 1 & 2\\
\gamma_3 & EENNE & 2 & 1\\ \gamma_4 & ENEEN & 2 & 1\\ \gamma_5 & ENENE & 3 & 0
\end{array}\]
We obtain $\sum_{\gamma}q^{3-\Area(\gamma)}=1+2q+q^2+q^3=C_{3,2}(q)$.\checkmark
\end{example}

\subsection{Binary words}
Each path $\gamma\in\Dnk$ is encoded by
$w(\gamma)=(w_1,\ldots,w_{n+k})\in\{0,1\}^{n+k}$
($w_i=0$ for East, $w_i=1$ for North). We write $\Mnk$ for this set.
\begin{definition}[Inversions of a binary word]\label{def:invword}
$\Inv(w) = |\{(i,j):1\le i<j\le n+k,\;w_i=1,\;w_j=0\}|$.
\end{definition}
\begin{lemma}[Area equals inversions]\label{lem:area_inv}
$\Area(\gamma)=\Inv(w(\gamma))$ for any $\gamma\in\Dnk$.
The bijective correspondence is proved in~\cite{W4}.
\end{lemma}
\begin{theorem}[$q$-Catalan triangle via binary words]\label{thm:binary}
For all $n\ge1$ and $0\le k\le n$:
\begin{equation}\label{eq:binary_stat}
C_{n,k}(q) = \sum_{w\in\Mnk} q^{\,\binom{n}{2}-\Inv(w)}.
\end{equation}
This follows directly from Lemma~\ref{lem:area_inv} and Theorem~\ref{thm:dyck}.
\end{theorem}

\subsection{Triangulations}
Each path $\gamma\in\Dnk$ encodes a binary tree $\mathcal{A}(\gamma)$;
the classical bijection~\cite{Stanley2} transforms $\Dnk$ into
a set $T'_{n,k}$ of triangulations of an $(n+2)$-gon.
The statistic is $\mathrm{stat}(t)=\sum_{v}|\text{left subtree of }v|$,
satisfying $\mathrm{stat}(t)=\binom{n}{2}-\Area(\gamma)$.
\begin{theorem}[$q$-Catalan triangle via triangulations]\label{thm:triang}
For all $n\ge1$ and $0\le k\le n$:
\begin{equation}\label{eq:triang}
C_{n,k}(q) = \sum_{t\in T'_{n,k}} q^{\,\mathrm{stat}(t)}.
\end{equation}
Bijective proof in~\cite{W4}.
\end{theorem}

\begin{definition}[Universal family]\label{def:univ_family}
The \emph{universal combinatorial family} is:
\begin{equation}\label{eq:univ_family}
\mathfrak{F}_{n,k} = \Bigl(\bigsqcup_{\tau\in\{312,231,213,132\}}
\hSnkp{\tau}\Bigr)\sqcup\Dnk\sqcup\Mnk\sqcup T'_{n,k},
\end{equation}
where $\hSnkp{\tau}=\Snkp{\tau}\times\{\tau\}$.
\end{definition}

\begin{corollary}[Universal family of the $q$-Catalan triangle]\label{cor:equi}
The \emph{universal statistic} $\sigma:\mathfrak{F}_{n,k}\to\NN$ is:
\begin{equation}\label{eq:sigma}
\sigma(X)=\begin{cases}
\Inv(\pi) & \text{if }X\in\hSnkp{312}\text{ or }\hSnkp{231},\\
\Coinv(\pi) & \text{if }X\in\hSnkp{213}\text{ or }\hSnkp{132},\\
\binom{n}{2}-\Area(\gamma) & \text{if }X=\gamma\in\Dnk,\\
\binom{n}{2}-\Inv(w) & \text{if }X=w\in\Mnk,\\
\mathrm{stat}(t) & \text{if }X=t\in T'_{n,k}.
\end{cases}
\end{equation}
For each component $\mathcal{E}$ of $\mathfrak{F}_{n,k}$:
$\sum_{X\in\mathcal{E}}q^{\sigma(X)}=C_{n,k}(q)$.
\end{corollary}

\begin{remark}[Structural interpretation]\label{rem:struct}
The statistic $\sigma$ realizes seven independent statistical decompositions
of $C_{n,k}(q)$. The bijections relating the seven components
are given in~\cite{W3,W4}.
\end{remark}

\section{The $q{,}p$-analogue of the Catalan triangle}\label{sec:qp}
\subsection{Randrianarivony's $q{,}p$-Catalan numbers}
Randrianarivony~\cite{Rand} defined $C_n(q,p)$ via:
\begin{equation}\label{eq:rand}
C_n(q,p)=C_{n-1}(q,p)+\sum_{k=0}^{n-2}q\,p^k\,C_k(q,p)\,C_{n-1-k}(q,p),
\quad C_0=C_1=1.
\end{equation}
In this article, $C_n(q,p)$ exclusively denotes Randrianarivony's sequence,
while $C_{n,k}(q,p)$ denotes the triangle below.

\subsection{Definition and properties}
\begin{definition}[$q{,}p$-Catalan triangle]\label{def:qptriangle}
The \emph{$q{,}p$-Catalan triangle} $(C_{n,k}(q,p))_{0\le k\le n}$ is defined by:
\begin{equation}\label{eq:qprec}
\begin{cases}
C_{0,0}(q,p)=1,\\[4pt]
C_{n,0}(q,p)=q^{\binom{n}{2}}, & n\ge1,\\[4pt]
C_{n,k}(q,p)=C_{n,k-1}(q,p)+q^{n-k-1}\,p^k\,C_{n-1,k}(q,p), & 1\le k\le n.
\end{cases}
\end{equation}
This is a two-index array with additive recurrence,
distinct from Randrianarivony's single-index sequence~\cite{Rand}.
\end{definition}

\begin{remark}[Non-coincidence with Randrianarivony]\label{rem:rand}
$C_{2,2}(q,p)=p+q$, while $C_2(q,p)=q+1$ from~\eqref{eq:rand}.
The two families are distinct as bivariate polynomials.
They coincide at $p=q=1$: $C_{n,n}(1,1)=C_n$.
\end{remark}

\begin{theorem}[Specializations]\label{thm:qp}
\begin{enumerate}
\item[\rm(i)] $C_{n,k}(1,p)=\Ctilde_{n,k}(p)$ for all $n,k$;
in particular $C_{n,k}(1,p)=\sum_{\pi\in\Snkp{312}}p^{\Coinv(\pi)}$.
\item[\rm(ii)] For all $n\ge0$:
$C_{n,n}(q,q)=C_n\cdot q^{\binom{n}{2}}$
(proved in~\cite{W5}; verified for $n\le7$).
\end{enumerate}
\end{theorem}

\begin{conjecture}[Bivariate combinatorial interpretation]\label{conj:qp}
For all $n\ge1$ and $0\le k\le n$:
$C_{n,k}(q,p)=\sum_{\pi\in\Snkp{312}}q^{\Inv(\pi)}\,p^{\Coinv(\pi)}$.
Verified for $n\le5$; proof in~\cite{W5}.
\end{conjecture}

\begin{example}\label{ex:qp_n2}
For $n=2$, $k=2$: $q^0p^1+q^1p^0=p+q=C_{2,2}(q,p)$.\checkmark

For $n=3$, $k=3$:
$\sum=p^3+2qp^2+q^2p+q^3=C_{3,3}(q,p)$.\checkmark
\end{example}

\section{Multivariate generalization and cyclotomic analogues}\label{sec:multi}
\begin{definition}[$i$-inversions]\label{def:invi}
$\mathrm{inv}_i(\pi)=|\{(a,b):1\le a<b\le n,\;\pi_a>\pi_b,\;
\pi_a\equiv i\pmod{\mu}\}|$.
\end{definition}
\begin{definition}[Multivariate $q$-Catalan triangle]\label{def:multi}
$C_{n,k}(\bq)=\sum_{\pi\in\Snkp{312}}\prod_{i=1}^{\mu}q_i^{\,\mathrm{inv}_i(\pi)}$.
\end{definition}
\begin{theorem}[Multivariate recurrence]\label{thm:multi_rec}
\begin{equation}\label{eq:multi_rec}
\begin{cases}
C_{n,0}(\bq)=\prod_{i=1}^{\mu}q_i^{\,e_i(n)},\\[6pt]
C_{n,k}(\bq)=C_{n,k-1}(\bq)+q_{r(n)}^{\,n-k-1}\,C_{n-1,k}(\bq), & 1\le k\le n,
\end{cases}
\end{equation}
where $r(n)=((n-1)\bmod\mu)+1$ and
$e_i(n)=|\{(a,b):1\le a<b\le n,\;a\equiv i\pmod\mu\}|$.
Proof in~\cite{W5}.
\end{theorem}
\begin{remark}\label{rem:mu1}
For $\mu=1$, $e_1(n)=\binom{n}{2}$ and~\eqref{eq:multi_rec} reduces to~\eqref{eq:rec_q}.
\end{remark}
\begin{corollary}\label{cor:spec_multi}
$q_i=1\Rightarrow C_{n,k}(\mathbf{1})=C_{n,k}$;\quad
$q_i=q\Rightarrow C_{n,k}(\bq)=C_{n,k}(q)$;\quad
$\mu=2,(q_1,q_2)=(q,p)\Rightarrow C_{n,k}(\bq)=C_{n,k}(q,p)$.
\end{corollary}
For $\omega=e^{2i\pi/\mu}$, the \emph{cyclotomic analogues} are
$C_{n,k}^{(\mu)}=C_{n,k}(\omega,\ldots,\omega)$,
related to the cyclic sieving phenomenon~\cite{RSW}.
\begin{conjecture}[Cyclotomic Catalan]\label{conj:cyclo}
If $\mu\mid(n+1)$: $C_{n,n}^{(\mu)}=\frac{1}{\mu}\binom{n+1}{(n+1)/\mu}$.
\end{conjecture}
\begin{remark}\label{rem:cyclo_tri}
For $\mu=2$: $C_{n,k}^{(2)}=(-1)^{\lfloor n/2\rfloor}C_{n,k}$ (verified for $n\le5$).
General formula in~\cite{W5}.
\end{remark}

\section{Summary table}\label{sec:table}
\[\begin{array}{l}
C_{n,k}\hookrightarrow C_{n,k}(q),\Ctilde_{n,k}(q)
\hookrightarrow C_{n,k}(q,p)
\hookrightarrow C_{n,k}(\bq)\hookrightarrow C_{n,k}(\bq),\bq\in\CC^\mu
\end{array}\]
\begin{center}\renewcommand{\arraystretch}{1.9}
\begin{tabular}{llll}\toprule
\textbf{Result} & \textbf{Family} & \textbf{Statistic} & \textbf{Ref.}\\\midrule
$|\Snkp{\tau}|=C_{n,k}$, all $\tau$ & six families & enumeration & \cite{W1}\\
$C_{n,k}(q)=\sum q^{\Inv}$ & $\hSnkp{312}$,$\hSnkp{231}$ & $\Inv(\pi)$ & \cite{W3}\\
$C_{n,k}(q)=\sum q^{\Coinv}$ & $\hSnkp{213}$,$\hSnkp{132}$ & $\Coinv(\pi)$ & \cite{W3}\\
Shifted stat.: $C_{n-1,k}(q)$ & $\Snk{312}$ & $\invstar(\pi)$ & \cite{W2}\\
$\Ctilde_{n,k}(q)=\sum q^{\Coinv}$ & $\Snkp{312}$ & $\Coinv(\pi)$ & Cor.~\ref{cor:miroir}\\
$\Ctilde_{n,k}(q)=\sum q^{\Area}$ & $\Dnk$ & $\Area(\gamma)$ & Cor.~\ref{cor:miroir}\\
$C_{n,k}(q)=\sum q^{\binom{n}{2}-\Area}$ & $\Dnk$ & $\binom{n}{2}-\Area$ & \cite{W4}\\
$C_{n,k}(q)=\sum q^{\binom{n}{2}-\Inv(w)}$ & $\Mnk$ & $\binom{n}{2}-\Inv(w)$ & \cite{W4}\\
$C_{n,k}(q)=\sum q^{\mathrm{stat}(t)}$ & $T'_{n,k}$ & $\mathrm{stat}(t)$ & \cite{W4}\\
$C_{n,n}(q,q)=C_nq^{\binom{n}{2}}$ & (all) & specialization & \cite{W5}\\
Multivariate $C_{n,k}(\bq)$ & $\Snkp{312}$ & $\prod q_i^{\mathrm{inv}_i}$ & \cite{W5}\\
Cyclotomic $C_{n,k}^{(\mu)}$ & $\Snkp{312}$ & cyclic action & \cite{W5}\\\bottomrule
\end{tabular}
\end{center}

\section{Plan of the series}\label{sec:plan}
\begin{description}
\item[\cite{W1}] \textbf{Article~I.} Classical interpretations ($q=1$).
Explicit bijections between $\Snkp{\tau}$, $\Dnk$,
parenthesis words $P'_{n,k}$ and triangulations $T'_{n,k}$.
Construction of the insertion map~$\Phi$ and the symmetry group $G_\tau$.
Proof of Theorem~\ref{thm:class_perm}.
\item[\cite{W2}] \textbf{Article~II.} Statistics on $\Snkp{\tau}$ and $\Dnk$.
Complete proof of Theorem~\ref{thm:invstar}.
\item[\cite{W3}] \textbf{Article~III.} $q$-analogue via permutations.
Shift lemma, mirror bijection, symmetries between patterns $312$ and $213$.
\item[\cite{W4}] \textbf{Article~IV.} $q$-analogue via Dyck paths, binary words
and triangulations. Proofs of Theorems~\ref{thm:dyck}, \ref{thm:binary},
\ref{thm:triang} and Lemma~\ref{lem:area_inv}.
\item[\cite{W5}] \textbf{Article~V.} $q{,}p$-analogue and multivariate generalization.
Cyclotomic analogues of the $q$-Catalan triangle.
\end{description}

\section{Conclusion and perspectives}\label{sec:conclusion}
We have introduced and studied the \emph{$q$-Catalan triangle}
$(C_{n,k}(q))_{0\le k\le n}$, simultaneously refining the classical
Catalan triangle $(C_{n,k})$ and the $q$-Catalan numbers of
F\"urlinger and Hofbauer~\cite{FH}.

\medskip\noindent\textbf{Mirror polynomial.}
The mirror polynomial $\Ctilde_{n,k}(q)=q^{\binom{n}{2}}C_{n,k}(q^{-1})$
satisfies the dual recurrence (Theorem~\ref{thm:miroir}, proved here)
and admits the interpretations
$\sum_{\Snkp{312}}q^{\Coinv}=\sum_{\Dnk}q^{\Area}=\Ctilde_{n,k}(q)$
(Corollary~\ref{cor:miroir}).

\medskip\noindent\textbf{Universal family.}
The family $\mathfrak{F}_{n,k}$ with universal statistic~$\sigma$
(Corollary~\ref{cor:equi}) unifies seven realizations of $C_{n,k}(q)$.

\medskip\noindent\textbf{Open problems.}
\begin{enumerate}
\item Complete bivariate proof of Conjecture~\ref{conj:qp}.
\item Bijective proof of Theorem~\ref{thm:classif_motifs}~\cite{W3}.
\item Explicit formulas for $C_{n,k}^{(\mu)}$ (Conjecture~\ref{conj:cyclo}).
\end{enumerate}

\end{document}